\documentclass[a4paper,12pt]{article}
\usepackage[cp1250]{inputenc}
\usepackage[T1]{fontenc}
\usepackage{amsmath}
\usepackage{amssymb}
\usepackage{amsthm} 
\usepackage{textcomp}
\usepackage{wrapfig}
\usepackage{url}
\usepackage{fancyhdr}
\usepackage{geometry}
\usepackage{stmaryrd}
\usepackage{bbding}
\usepackage{bbm}
\everymath{\displaystyle}

\theoremstyle{plain} 
\newtheorem{tw}{Theorem}[section]	

\theoremstyle{definition} 

\newtheorem{example}{Example}[section]
\newtheorem{remark}{Remark}
\theoremstyle{remark} 

\newcommand{\sP}{\mathsf{P}}

\newcommand\cF{{\mathcal F}}
\newcommand\cM{{\mathcal M}}

\newcommand\md{{\,\mathrm{d}}}

\newcommand{\rM}{\mathrm{M}}

\renewcommand\c{\circ}

\newcommand\lo{{\,\lozenge\,}} 

\newcommand\loo{{\lozenge}}
\newcommand\vtr{\vartriangle}

\newcommand\intl{\int\limits}
\newcommand\nint{(N)\int}
\newcommand\nintl{(N)\int\limits}
\newcommand\sint{(S)\int}
\newcommand\sintl{(S)\int\limits}
\newcommand\cint{\int}
\newcommand\cintl{\int\limits}

\renewcommand\ge{\geqslant}
\renewcommand\le{\leqslant}

\title{On Carlson's inequality for Sugeno and Choquet integrals}
\author{Michał Boczek\footnote{Corresponding author. E-mail adress: 800401@edu.p.lodz.pl;}, Marek Kaluszka
\\
{\emph{
\small{Institute of Mathematics, Lodz University of Technology, 90-924 Lodz, Poland}}}}
\date{}

\begin{document}
\maketitle


\begin{abstract}
We present
a
Carlson type inequality for the generalized
Sugeno
integral and
a~much wider
class of functions than the comonotone functions.
We also provide three  Carlson type inequalities for the Choquet integral.
Our inequalities generalize many known results.
\end{abstract}

{\it Keywords: }{Choquet integral; Sugeno integral; Capacity; Semicopula; Carlson inequality.}

\section{Introduction}
The pioneering concept of the fuzzy integral was introduced by Sugeno $\cite{sugeno1}$
as a~tool for modelling non-deterministic problems. Theoretical investigations of the integral
and its generalizations have been pursued by many researchers.
Wang and Klir $\cite{wang}$ presented an excellent general overview on fuzzy integration theory. On the other hand,
fuzzy integrals have also been successfully applied to various fields
(see, e.g., $\cite{Hu,Narukawa}$).

The study of inequalities for Sugeno integral was initiated by Rom\'an-Flores et al. $\cite{Flores5}.$
Since then, the fuzzy integral counterparts of several classical inequalities, including Chebyshev's,
Jensen's, Minkowski's and H\"older's inequalities, are given by Flores-Franuli\v c and Rom\'an-Flores $\cite{Flores4}$,
Agahi et al. $\cite{agahi9},$  L. Wu et al. $\cite{wu1}$ and
others.
Furthermore many researchers started to study inequalities for the seminormed Sugeno integral $\cite{agahi3,boczek1,boczek2,ouyang1}.$

The  Carlson  inequality for the Lebesgue
integral is of the form
\begin{align}\label{c1}
\intl_0^\infty f(x)\md x\le \sqrt{\pi}\bigg(\intl_0^\infty f^2(x)\md x\bigg)^{\tfrac{1}{4}}\bigg(\intl_0^\infty x^2f^2(x)\md x\bigg)^{\tfrac{1}{4}},
\end{align}
where $f$ is any non-negative, measurable function such that the integrals
on the right-hand side converge.
The equality in $\eqref{c1}$ is attained iff $f(x)=\frac{\alpha}{\beta+x^2}$ for some constants $\alpha\ge 0$, $\beta>0.$
The modified versions of the Carlson inequality can
be
found in $\cite{barza}$ and $\cite{mit}.$

The purpose of this paper is to study the Carlson inequality for
the generalized Sugeno as well the Choquet integrals.
In Section $2,$ we provide inequalities for the generalized Sugeno integral.
The results are obtained for a~rich class of functions, including the comonotone functions as a~special case.
In 
Section $3$ we present
the corresponding
results for the Choquet integral.

\section{Carlson's type inequalities for Sugeno integral}

Let $(X,\cF)$  be a~measurable space and $\mu\colon \cF \to Y$ be a~monotone measure,
i.e., $\mu(\emptyset)=0$, $\mu (X)>0$ and $\mu(A)\le\mu(B)$ whenever $A\subset B.$ Throughout the paper $Y=[0,1]$ or $Y=[0,\infty].$
Suppose  $\circ\colon Y\times Y\to Y$ is a~non-decreasing operator, i.e. $a\circ c\ge b\circ d$ for $a\ge b$ and $c\ge d.$
An 
example of non-decreasing
operators
is a~$t$-{\it seminorm},  also called
a~{\it semicopula}   $\cite{durante1, ouyang1}.$
There are three important $t$-seminorms: $\rM,$ $\Pi$ and $\c_L,$ where $\rM(a,b)=a\wedge b,$ $\Pi(a,b)=ab$ and $\c_L(a,b)=(a+b-1)\vee 0$ usually called the {\it Łukasiewicz t-norm}  $\cite{klement2}.$ Hereafter,
$a\vee b=\max(a,b)$ and $a\wedge b=\min(a,b).$

For a~measurable function $h\colon X\to Y,$ 
we  define the {\it generalized Sugeno integral} 
of $h$ on a~set $A\in \cF$ with respect to $\mu$ and
a~non-decreasing operator $\c\colon Y\times Y \to Y$ as
\begin{align}\label{p1}
\intl_A h\c \mu=\sup_{\alpha\in Y} \left\{\alpha\c \mu\big(A\cap\lbrace h\ge \alpha \rbrace\big)\right\},
\end{align}
where $\left\{h\ge a\right\}$ stands for $\left\{x\in X\colon h(x)\ge a\right\}$.
For $\c=\rM$, we get the {\it Sugeno integral} $\cite{sugeno1}.$ If  $\c=\Pi,$
then $\eqref{p1}$ is called the {\it Shilkret integral}\; $\cite{shilkret}.$
We denote the Sugeno and the Shilkret integral as $\sint_A f\md\mu$ and $\nint_A f\md\mu$, respectively. Moreover,  we obtain
the
{\it seminormed fuzzy integral}  if $\c$ is a~semicopula $\cite{suarez}.$

Let $f,g\colon X\to Y$ be measurable functions and $A,B\in\cF$. The  functions $f_|{_A}$ and $g_|{_B}$  are {\it positively dependent} with respect to $\mu$ and an operator  $\vtr\colon Y\times Y\to Y$ if for any $a,b\in Y$
\begin{align}\label{p2}
\mu\Big(\left\{f_|{_A}\ge a\right\}\cap\left\{g_|{_B}\ge b\right\}\Big)\ge\mu\Big(\left\{f_|{_A}\ge a\right\}\Big)\vtr \mu\Big(\left\{g_|{_B}\ge b\right\}\Big),
\end{align}
where $h_|{_C}$ denotes the restriction of the function $h\colon X\to Y$ to a~set $C\subset X.$ Obviously, $\left\{h_|{_C}\ge a\right\}=\left\{x\in C\colon h(x)\ge a\right\}=C\cap\left\{h\ge a\right\}.$
Taking
$a\vartriangle b=a\wedge b$ and $a\vartriangle b=ab,$ we recover two important examples of positively dependent functions, namely comonotone functions and independent random variables. Recall that
$f$ and $g$ are comonotone if $(f(x)-f(y))(g(x)-g(y))\ge 0$ for all $x,y\in X$. More examples of positively dependent functions can be found in $\cite{boczek1}.$

Suppose  $\star,\Box\colon Y\times Y\to Y$ are non-decreasing operators. Let $\loo\colon Y\times  Y \to Y$ be a~non-decreasing and left-continuous operator, i.e. $\lim_{n\to\infty} (x_n \c y_n)=x\c y$ for all $x_n \nearrow x$ and $y_n \nearrow y,$ where $a_n\nearrow a$ means that $\lim_{n\to\infty} a_n=a$ and $a_n<a_{n+1}$ for all $n.$ Assume $\vtr\colon Y\times Y\to Y$ is an~arbitrary operator and $f,f_i,g\colon X\to Y,$ $i=1,2,3,$ are measurable functions.

We recall two inequalities for generalized Sugeno integral.

\begin{tw}[$\cite{boczek2}$] For $s\ge 1$ and $A\in\cF,$
the Jensen type inequality
\begin{align}\label{p3}
\intl_A f^s\c \mu\ge \bigg(\intl_A f\c \mu\bigg)^s
\end{align}
holds if $a^s\c b\ge (a\c b)^s$ for all $a,b\in Y.$
\end{tw}

\begin{remark}\label{rem1} If $\circ=\wedge$, then $\eqref{p3}$ is satisfied provided
$\sint_A f\md\mu\le 1$ (see $\cite{xu}$ and $\cite{boczek1},$ Theorem $3.1$).
\end{remark}

\begin{tw}[$\cite{boczek1}$]
The Chebyshev type inequality
of the form
\begin{align}\label{p4}
\intl_{A\cap B} \big( f_1\,\Box\, f_2\big) \c \mu\ge \bigg(\intl_A f_1\c \mu\bigg) \lo \bigg(\intl_B f_2\c \mu\bigg)
\end{align}
holds for all positively dependent functions
${f_1}_|{_A},$ ${f_2}_|{_B}$ and $A,B\in \cF$  if
$\big(a\,\Box\, b\big)\c \big(c\vtr d\big)\ge \big(a\c c\big) \lo \big(b\c d\big)$
for all $a,b,c,d\in Y.$
\end{tw}

Now we are ready to derive a~Carlson-type inequality for the generalized Sugeno integral.

\begin{tw}\label{tw1}
Suppose  $p,q\ge 1$ and $r,s>0.$ Then for arbitrary pairs of positively dependent functions
$f_{|{A}},$ $g_{|{B}}$ and  $f_{|{A}},$ $h_{|{B}}$,
the following inequality
\begin{align}\label{nowy1}
\bigg(\Big(\intl_A f\c \mu\Big) \lo \Big(\intl_B g\c\mu\Big)\bigg)^r\star& \bigg(\Big(\intl_A f\c \mu\Big) \lo \Big(\intl_B h\c\mu\Big)\bigg)^s\\&\le \bigg(\intl_{A\cap B} \big(f\,\Box\, g\big)^p\c \mu \bigg)^{\tfrac{r}{p}}\star\bigg(\intl_{A\cap B} \big(f\,\Box\, h\big)^q\c \mu \bigg)^{\tfrac{s}{q}}\nonumber
\end{align}
is satisfied if for all $a,b,c,d\in Y$ and $s>1$,
\begin{align}\label{warunki}
a^s\c b\ge \big(a\c b\big)^s,\quad
\big(a\,\Box\, b)\c \big(c\vtr d\big)
\ge \big(a\c c\big) \lo \big(b\c d\big).
\end{align}
\end{tw}
\begin{proof} Observe that all integrals in $\eqref{nowy1}$
are elements of $Y$.
From the Jensen  inequality $\eqref{p3},$  it follows that
\begin{align}
\intl_{A\cap B} \big(f\,\Box \,g\big) \c \mu&\le \bigg(\intl_{A\cap B} \big(f\,\Box\, g\big)^p\c \mu \bigg)^{\tfrac{1}{p}}, \label{n1}\\
\intl_{A\cap B} \big(f\,\Box\, h\big)\c\mu &\le \bigg(\intl_{A\cap B} \big(f\,\Box\, h\big)^q\c \mu\bigg)^{\tfrac{1}{q}}.\label{n2}
\end{align}
The operator $\star$ is  non-decreasing, so by $\eqref{n1}$ and $\eqref{n2},$
\begin{align}\label{n3}
\bigg(\intl_{A\cap B} \big(f\,\Box\, g\big) \c \mu\bigg)^r&\star \bigg(\intl_{A\cap B} \big(f\,\Box\, h\big)\c\mu \bigg)^s\nonumber\\&\le \bigg(\intl_{A\cap B} \big(f\,\Box\, g\big)^p\c \mu \bigg)^{\tfrac{r}{p}}\star\bigg(\intl_{A\cap B} \big(f_1\,\Box\, h\big)^q\c \mu \bigg)^{\tfrac{s}{q}}.
\end{align}
From $\eqref{p4}$ we get
\begin{align}
\intl_{A\cap B} \big(f\,\Box\, \psi\big) \c \mu&\ge \bigg(\intl_A f\c \mu\bigg) \loo \bigg(\intl_B \psi \c \mu\bigg)\;\;\hbox{for}\; \psi=g,h.\label{n4}
\end{align}
To complete the proof, it is enough to apply $\eqref{n4}$ to $\eqref{n3}.$
\end{proof}

Theorem $\ref{tw1}$ extends  all known
(obtained by different methods)
Carlson type inequalities for the Sugeno integral.
In order to see this, we first put
$A=B,$ $\vtr=\c=\wedge$ and $\Box=\star=\lo=\cdot$ in Theorem $\ref{tw1}.$
Putting further $g=1,$ $h=x,$  $p=q=2,$ $r=s=1$ and $A=[0,1]$ yields
the~result of Caballero et al. $\cite{caballero4}.$
If $\mu$ is the Lebesgue measure then
\begin{align*}
\sintl_A f\md\mu \le \sqrt{2} \bigg(\sintl_A  f\md\mu\bigg)^{\tfrac{1}{4}} \bigg(\sintl_A x^2f^2\md\mu\bigg)^{\tfrac{1}{4}},
\end{align*}
since $\sint_{[0,1]} x\md\mu=0.5$ and  if $f$ and $g$ are comonotone, then  $f_{|{A}}$ and
$g_{|{A}}$ are positively dependent with respect to the operator $\wedge$ (see Example $2.1$ in $\cite{boczek1}$).

Setting $r=s=1$, we obtain
Theorem $3.1$ of  Xu and  Ouyang $\cite{xu}$
\begin{align*}
\sintl_A f\md\mu \le\frac{1}{\sqrt{C}} \bigg(\sintl_A f^p g^p \md\mu\bigg)^{\tfrac{1}{2p}}\bigg(\sintl_A f^q h^q \md\mu\bigg)^{\tfrac{1}{2q}},
\end{align*}
where $C=\Big(\sint_A g\md\mu\Big)\Big(\sint_A h\md\mu\Big)$ (see Remark
$\ref{rem1}).$
Taking  $r=p/(p+q)$ and $s=1-r$, we get  Theorem $2.7$ from $\cite{xwang}$
\begin{align*}
\sintl_A f\md\mu\le \frac{1}{K} \bigg(\sintl_A f^p g^p\md\mu\bigg)^{\tfrac{1}{p+q}}\bigg(\sintl_A f^q h^q\md\mu\bigg)^{\tfrac{1}{p+q}},
\end{align*}
where $$K=\bigg(\sintl_A g\md\mu\bigg)^{\tfrac{p}{p+q}}\bigg(\sintl_A h\md\mu\bigg)^{\tfrac{q}{p+q}}.$$
Combining the above results with
other inequalities for
comonotone functions one can also derive
(similarly as in $\cite{daraby}$)
some related Carlson type inequalities for the Sugeno integral.

From Theorem $\ref{tw1}$ one can obtain many
other Carlson type inequalities since the conditions $\eqref{warunki}$ are fulfilled by
many systems of operators. Examples are:
\begin{enumerate}
\item  $\vtr=\wedge$ and $\Box=\lo=\circ,$ where  $\circ$ is any $t$-norm satysfying
the
condition $(a^s\circ b)\ge (a\circ b)^s$ for $s\ge 1$ since $a\circ b\le a\wedge b$ and any $t$-norm is
an
associative and commutative operator $\cite{klement2}$;
\item $\vtr=\Box=\circ=\lo=\cdot$ on $Y=[0,1]$;
\item  $\vtr=\Box=\lo=\cdot$ and $\circ=\wedge$ with  $Y=[0,1]$;
\item $\vtr=\Box=\lo,$ $\circ=\wedge$ and
 $Y=[0,1]$;
\item $\vtr=\Box=\lo=\circ,$ where $\circ$  is any $t$-norm satysfying
the
condition $(a^s\circ b)\ge (a\circ b)^s$ for $s\ge 1$, e.g.  the Dombi $t$-norm
$a\circ b=ab/(a+b-ab);$
\item  $\Box=\lo,$ $\vtr$ is any operator, $a\circ b=a$ for all $a,b\in Y$ and $Y=[0,1]$ or $Y=[0,\infty].$
\end{enumerate}

\begin{example}
The following inequality for
the Shilkret integral of a~non-decreasing function $f$ is valid:
\begin{align*}
\nintl_A f\md\mu \le \frac{1}{\sqrt{K}} \cdot \bigg(\nintl_A  f^2\md\mu\bigg)^{\tfrac{1}{4}} \bigg(\nintl_A x^2f^2\md\mu\bigg)^{\tfrac{1}{4}},
\end{align*}
where $K=\mu(A)\cdot \Big(\nint_A x\md\mu\Big);$
to see this  put $\vtr=\wedge$ or  $\vtr=\cdot$ , $g=1,$ $h=x,$ $\loo=\star=\Box=\c=\cdot,$ $p=q=2,$ $r=s=1$ and  $A=B$ in Theorem $\ref{tw1}.$
\end{example}

\begin{example}
Let  $(X,\cF,\sP)$ be a~probability space. Put   $Y=[0,1],$  $r=s=1,$ $g=1,$ $A=B=X,$ $f=\phi \big(U)$
and $h=1-\psi (U)$, where  $U$ has the~uniform distribution on $[0,1]$ and $\phi,\psi\colon [0,1]\to [0,1]$ are increasing functions. The functions $f$ and $h$ are not comonotone but
\begin{align*}
\sP\big(f\ge a, h\ge b\big)&=\big(\psi ^{-1}(1-b)-\phi^{-1}(a))_+=\sP\big(f\ge a\big)\c_L\sP\big(h\ge b\big),
\end{align*}
so $f$ and $h$ are positively dependent with respect to $\sP$ and $\c _L$.
The conditions $\eqref{warunki}$ are satisfied for  $\vtr=\Box=\loo=\c_L$ and $\star=\circ=\cdot$ (see $\cite{boczek1},$ formula $(40)$),
thus the
corresponding
Carlson inequality takes  the form
 \begin{align}
\big(N(f)\c_L 1\big)& \cdot \big(N(f) \c_L N(h)\big)\nonumber\le  \big(N(f^p)\big)^{\tfrac{1}{p}}\Big(N\big((f \c_L h)^q\big)\Big)^{\tfrac{1}{q}}
,
\end{align}
where $N(f)=\nintl_X f \md\sP$.
\end{example}

\section{Carlson's type inequality for Choquet integral}

In this section, $\mu\colon\cF \rightarrow [0,\infty ]$ is a~monotone measure. Denote by $\cM$ the set of monotone measures on $(X,\cF).$
The Choquet integral of $f\colon X\to [0,\infty)$ on $A\in\cF$ is defined as
\begin{align*}
\cintl_A f\md \mu =\intl_0^\infty \mu\big(A\cap \lbrace  f\ge t\rbrace\big) \md t,
\end{align*}
where  the integral on the right-hand side  is the improper Riemann integral. A~function $f$ is said to be {\it integrable} on a~measurable set $A$ if $\cint _Af\md\mu<\infty.$ The importance of the Choquet integral still increases due to many applications in mathematics and economics, see for instance
$\cite{cerda, den, gra, hei, ka}.$

First, we show that it does not exist  a~functional $c\colon\cM\to [0,\infty]$ such that for any monotone measure $\mu$ and any integrable function $f$
\begin{align}\label{lab1}
\intl_X f\md \mu \le c(\mu)\biggl(\intl_Xgf^2\md \mu\biggr)^{\frac{1}{4}}\biggl(\intl_Xhf^2\md \mu\biggr)^{\frac{1}{4}}
\end{align}
provided $\inf _{x\in X}(g(x)h(x))=0.$
Indeed, put $\mu (A)=1$ for all $A\neq \emptyset,$ $f(x)=1$ for $x=t$ and $f(x)=0$ otherwise, where $t$ is any fixed point of $X.$  Since
$\int _X\psi \md \mu =\sup _{x\in X}\psi (x)$, from $\eqref{lab1}$  we have $1\le c(\mu )(g(t)h(t))^{1/4},$ a~contradiction with
$\inf _{x\in X}(g(x)h(x))=0$ and $c(\mu)<\infty.$ 
Therefore, some extra conditions should be imposed on $f,g,h$ or $\mu$.

Now, we present the  Carlson type inequality for
the Choquet integral of comonotone functions.
\medskip

\begin{tw}\label{tw2} Let  $p,q\ge 1$ and $r,s>0.$ Suppose
$f,g\colon X\to [0,\infty)$ and $f,h \colon X\to [0,\infty)$ are
pairs of comonotone functions. If $f$ is integrable on  $A,$ then
\begin{align}\label{nie1}
\cintl_A f\md\mu \le K\big(\mu (A)\big)^d\bigg(\cintl_A f^p g^p\md\mu\bigg)^{\tfrac{r}{p(r+s)}}\bigg(\cintl_A f^q h^q\md\mu\bigg)^{\tfrac{s}{q(r+s)}},
\end{align}
where $K=\Big(\cint _A g\md\mu\Big)^{-\tfrac{r}{r+s}}\Big(\cint _Ah\md\mu\Big)^{-\tfrac{s}{r+s}}$ and $d=2-\tfrac{1}{r+s}\left(\tfrac{r}{p}+\tfrac{s}{q}\right).$
\end{tw}
\begin{proof} Without loss of generality, we assume that    $0<\mu (A)<\infty$.
 Put $m(B)=\mu (A\cap B)/\mu (A)$ for $B\in \cF.$
For a
given $c\ge 1$, the
following
Jensen type inequality
\begin{align}\label{nc2}
\bigg(\intl f \md m\bigg)^c\le \intl f^c \md m 
\end{align}
is satisfied
$\cite{girotto2, mes, zao}$.  Hereafter,  we write $\int f\md m$ instead of $\cint _Af\md m$.
From  $\eqref{nc2},$ we have
\begin{align}\label{new1}
\bigg(\int fg \md m\bigg)^r\bigg(\int fh \md m\bigg)^s\le \bigg(\int f^p g^p
\md m\bigg)^\frac{r}{p} \bigg(\int f^q h^q \md m\bigg)^\frac{s}{q}.
\end{align}
Since $f,g$ are comonotone functions, the following Chebyshev inequality
\begin{align*}
\int fg\md m \ge \int f\md m \int g\md m
\end{align*}
holds (see $\cite{girotto1}$).
The functions $f, h$ are also comonotone, so from
$\eqref{new1}$ we get
\begin{align*}
\bigg(\int f\md m\bigg)^{r+s}\bigg(\int g\md m\bigg)^r\bigg(\int h \md m\bigg)^s\le  \bigg(\int f^p g^p \md m\biggr)^\frac{r}{p}
\bigg(\int f^q h^q \md m\bigg)^\frac{s}{q}.
\end{align*}
Combining this with the equality
$\int \phi\md m=\big(\mu (A)\big)^{-1}\cint _A \phi \md\mu$,
completes the proof.
\end{proof}

Putting $g=1$ and $r=s$ in Theorem $\ref{tw2},$ we have
\begin{align*}
\bigg(\cintl_A f\md\mu \bigg)^2\le
\frac{\mu (A)^{3-\left(\tfrac{1}{p}+\tfrac{1}{q}\right)}}{\cintl_A h \md \mu}\bigg(\cintl_A f^p\md\mu\bigg)^{\tfrac{1}{p}}\bigg(\cintl_A f^q h^q\md\mu\bigg)^{\tfrac{1}{q}},
\end{align*}
since $\cint_A 1\md\mu =\mu(A)$. This result was obtained by  Ouyang for $p,q>1$  as a~consequence of H\"older's inequality for
the
Choquet integral
of comonotone functions $f,h$ and Chebyshev's inequality
$\cite{Ou1,Ou}.$

The inequality $\eqref{nie1}$ is sharp. In fact, if
$\mu(B)=1$ for $B\neq \emptyset$, then
$\int _A\phi \md\mu=s(\phi)$, where $s(\phi)$ denotes the supremum of $\phi$ on $A,$ so the inequality $\eqref{nie1}$ takes the form
\begin{align}\label{nowu1}
s(f)\le s(g)^{-\frac{r}{r+s}}s(h)^{-\frac{s}{r+s}}\bigl(s(fg)\bigr)^{\frac{r}{r+s}}\bigl(s(fh)\bigr)^{\frac{s}{r+s}}.
\end{align}
Since $s(\phi\psi)=s(\phi)s(\psi)$ for comonotone functions $\phi$, $\psi$ (see $\cite{mu}$),
the equality in $\eqref{nowu1}$ is attained.

Now, we provide  the Carlson type inequality for the Choquet integral with respect to a~submodular monotone measure $\mu.$
Recall that $\mu$ is  {\it submodular} if
\begin{align*}
\mu(A\cap B)+\mu(A\cup B)\le \mu(A)+\mu(B)
\end{align*}
for $A,B\in\cF.$ The Choquet integral is subadditive for all measurable functions $f,g$ iff $\mu$ is  submodular
(see $\cite{pap1},$ Theorem $7.7$).
Define
\begin{align}\label{wzorH}
H_{pq}(a,b)=(ab)^{\tfrac{1}{p}}\biggl(\cintl_A \frac{1}{(bg+ah)^{q-1}}\md\mu\biggr)^{\tfrac{1}{q}}.
\end{align}

\begin{tw}\label{twierdzenie3}
If $\mu$ is submodular, $f\colon X\to [0,\infty)$ and $A\in \mathcal{F},$ then
\begin{align}\label{nc3x}
\cintl_A f\md\mu\le 2^{1/p}H_{pq}\bigg(\cintl_A gf^p\md\mu,\cintl_A hf^p\md\mu \bigg),
\end{align}
where  $p>1$
and $1/p+1/q=1.$ The equality in $\eqref{nc3x}$ is attained if
$
\biggl(g\intl_Ahf^p\md\mu+h\intl_Agf^p\md\mu\biggr)^qf^p=\gamma
$
for some $\gamma\ge 0$ provided $\mu $ is modular or $gf^p$ and $hf^p$ are comonotone functions.
\end{tw}
\begin{proof}
Since $\mu$ is submodular,
the following H\"older inequality
\begin{align}\label{nc2a}
\cintl_A\phi \psi \md\mu \le \bigg(\cintl_A\phi^p\md\mu \bigg)^{\tfrac{1}{p}}\bigg(\cintl_A\psi^q\md\mu \bigg)^{\tfrac{1}{q}}
\end{align}
is valid,
where $\phi,\psi\ge 0$ (see $\cite{rwang1},$ Theorem $3.5$). The equality in $\eqref{nc2a}$
holds if $\alpha \phi^p=\beta \psi^q$  for $\alpha,\beta\ge 0$, $\alpha+\beta >0$ $\cite{nikulescu}.$
By
$\eqref{nc2a}$ and the
subadditivity and positively homogeneity of the Choquet integral,
we get
\begin{align}\label{additional}
\cintl_A f\md\mu&=\cintl_A (bg+ah)^{1/p}f\frac{1}{(bg+ah)^{1/p}}\md\mu\nonumber\\
&\le \bigg(\cintl_A (bg+ah) f^p\md\mu\bigg)^{\tfrac{1}{p}}\bigg(\cintl_A \frac{1}{(bg+ah)^{q/p}}\md\mu\bigg)^{\tfrac{1}{q}}\nonumber\\
&\le \bigg(b\cintl_A gf^p\md\mu+a\cintl_A hf^p\md\mu\bigg)^{\tfrac{1}{p}}\bigg(\cintl_A \frac{1}{(bg+ah)^{q-1}}\md\mu\bigg)^{\tfrac{1}{q}}.
\end{align}
Putting $a=\cint_A gf^p\md\mu$ and $b=\cint_A hf^p\md\mu,$
we obtain $\eqref{nc3x}.$ Note that if $\mu$ is modular then from Theorem $7.7$ of
$\cite{pap1}$ it follows that
\begin{equation*}
\cintl_A (bgf^p+ahf^p)\md\mu=\cintl_A bgf^p\md\mu+\cintl_A ahf^p\md\mu.\qedhere
\end{equation*}

\end{proof}

If $\mu$ is the Lebesgue measure, $g(x)=1$, $h(x)=x^2$, $p=q=2$ and $A=[0,\infty],$ we obtain the classical Carlson inequality  $\eqref{c1}.$

Next, we present the Carlson type inequality for the Choquet integral with respect to a~subadditive monotone measure $\mu.$

\begin{tw}
Suppose $f\colon X\to [0,\infty)$, $A\subset [0,\infty]$ and $\mu\in\cM$ such that $\mu(A\cup B)\le \mu(A)+\mu(B)$ for $A,B\in\cF.$  Then
\begin{align*}\label{nc3}
\cintl_A f\md\mu\le 4^{1/p} \biggl(\frac{1}{\sqrt{p}}+\frac{1}{\sqrt{q}}\biggr)^2 H_{pq}\bigg(\cintl_A gf^p\md\mu,\cintl_A hf^p\md\mu \bigg),
\end{align*}
where $H_{pq}(a,b)$ is given by $\eqref{wzorH},$  $p>1$
and $1/p+1/q=1.$
\end{tw}

\begin{proof} The proof is similar to that of Theorem $\ref{twierdzenie3},$
but
we use the following inequalities
(see $\cite{Shi}$)
\begin{align*}
\cintl_Afg\md\mu \le \biggl(\frac{1}{\sqrt{p}}+\frac{1}{\sqrt{q}}\biggr)^2\bigg(\cintl_Af^p\md\mu \bigg)^{\tfrac{1}{p}}\bigg(\cintl_Ag^q\md\mu \bigg)^{\tfrac{1}{q}},
\end{align*}
\begin{align*}
\cintl_A(f+g)\md\mu \le 2\biggl(\cintl_Af\md\mu+\cintl_Ag\md\mu\biggr).
\end{align*}
instead of those in $\eqref{additional}.$
Since $(1/\sqrt{p}+1/\sqrt{q})^2\le 2,$
the bounds obtained are better than the bounds of $\cite{cerda}$.
\end{proof}




\addcontentsline{toc}{chapter}{Literatura}















\end{document}